\documentclass[a4paper]{amsart}
\usepackage{amsmath,amsthm,amssymb,latexsym,epic,bbm,comment}
\usepackage{graphicx,enumerate,stmaryrd}
\usepackage[all,2cell]{xy}
\xyoption{2cell}

\usepackage[active]{srcltx}
\usepackage[parfill]{parskip}
\UseTwocells

\newtheorem{theorem}{Theorem}
\newtheorem{lemma}[theorem]{Lemma}

\newtheorem{proposition}[theorem]{Proposition}

\newtheorem{conjecture}[theorem]{Conjecture}

\font\sc=rsfs10
\newcommand{\cC}{\sc\mbox{C}\hspace{1.0pt}}

\newcommand{\cP}{\sc\mbox{P}\hspace{1.0pt}}

\font\scc=rsfs7
\newcommand{\ccC}{\scc\mbox{C}\hspace{1.0pt}}

\begin{document}

\title[Parabolic projective functors in type $A$]
{Parabolic projective functors in type $A$}
\author{Tobias Kildetoft and Volodymyr Mazorchuk}

\begin{abstract}
We classify projective functors on the regular block of 
Rocha-Caridi's parabolic version of the BGG category $\mathcal{O}$
in type $A$. In fact, we show that, in type $A$, the restriction of an
indecomposable projective functor from $\mathcal{O}$ to the
parabolic category is either indecomposable or zero.
As a consequence, we obtain that projective functors on 
the parabolic category $\mathcal{O}$ in type $A$ are
completely determined, up to isomorphism, by the linear 
transformations they induce on the level of the 
Grothendieck group, which was  conjectured by Stroppel in \cite{St}.
\end{abstract}

\maketitle

\section{Introduction and description of the results}\label{s0}

Category $\mathcal{O}$ associated to a fixed triangular decomposition 
$\mathfrak{g}=\mathfrak{n}_-\oplus \mathfrak{h}\oplus\mathfrak{n}_+$ of a
semi-simple complex finite dimensional Lie algebra $\mathfrak{g}$ was 
introduced in \cite{BGG}. For each parabolic subalgebra $\mathfrak{p}$ of
$\mathfrak{g}$ containing $\mathfrak{h}\oplus\mathfrak{n}_+$, there is a parabolic
version $\mathcal{O}^{\mathfrak{p}}$ of $\mathcal{O}$ introduced in \cite{RC}.
An important role in understanding the structure of both $\mathcal{O}$ and
$\mathcal{O}^{\mathfrak{p}}$ is played by the so-called {\em projective functors},
that is endofunctors of these categories isomorphic to direct summands of 
tensoring with finite dimensional $\mathfrak{g}$-modules. Indecomposable 
projective functors on $\mathcal{O}$ were classified, in terms of the
action of the Weyl group $W$ of $\mathfrak{g}$ on $\mathfrak{h}^*$, in \cite{BG}. 

Let $\mathcal{O}_0$ denote the principal block of $\mathcal{O}$, that is the
indecomposable direct summand of $\mathcal{O}$ containing the trivial $\mathfrak{g}$-module. 
Formulated in modern terms, the main result of \cite{BG} asserts that the
action of projective functors on $\mathcal{O}_0$ categorifies, using the 
Grothendieck group decategorification, the right regular representation
of $W$, see \cite[Lecture~5]{Ma} for details. In particular, isomorphism classes
of indecomposable projective functors on $\mathcal{O}_0$ turn out to be in a 
natural bijection with elements in $W$ and have a nice combinatorial description
in terms of Kazhdan-Lusztig combinatorics from \cite{KaLu}, see \cite[Lecture~7]{Ma} for details. 

In the case of a Lie algebra of type $A$, the action of projective functors on 
$\mathcal{O}$ and, especially, on $\mathcal{O}^{\mathfrak{p}}$ for a {\em maximal}
parabolic subalgebra $\mathfrak{p}$ plays a crucial role in the category 
$\mathcal{O}$ reformulation, given in \cite{St}, of Khovanov homology for oriented links,
originally defined in \cite{Kh}. In particular, the paper \cite{St} establishes the 
following two properties for projective functors on $\mathcal{O}^{\mathfrak{p}}$ for
a {\bf maximal parabolic} subalgebra $\mathfrak{p}$ in type $A$:
\begin{itemize}
\item The restriction of an indecomposable projective functor from $\mathcal{O}_0$ to
$\mathcal{O}_0^{\mathfrak{p}}$ is either indecomposable or zero, see \cite[Theorem~5.1]{St}.
\item A projective functor on $\mathcal{O}_0^{\mathfrak{p}}$ is completely determined, up to
isomorphism, by the linear  transformation it induces on the level of the 
Grothendieck group, see \cite[Theorem~5.7]{St}.
\end{itemize}
Moreover, it is conjectured in \cite[Conjecture~3.3]{St} that the second property should
hold in the general case (note that the first property fails outside
type $A$, see, for example,  \cite[Example~3.7(c)]{St}). The action of projective functors 
on arbitrary parabolic categories in type $A$ is used for
categorification of other quantum link invariants, see \cite{MS3}.
The aim of the present paper is to prove \cite[Conjecture~3.3]{St} for any parabolic
category $\mathcal{O}_0^{\mathfrak{p}}$, not necessarily a maximal one, in type $A$. 
Our main result is the following theorem.

\begin{theorem}\label{thm1}
Let $\mathfrak{g}=\mathfrak{sl}_n(\mathbb{C})$ and $\mathfrak{p}$ be any parabolic subalgebra of 
$\mathfrak{g}$ containing $\mathfrak{h}\oplus\mathfrak{n}_+$. Then we have the following: 
\begin{enumerate}[$($i$)$]
\item\label{thm1.1} The restriction of an indecomposable projective functor from 
$\mathcal{O}_0$ to $\mathcal{O}_0^{\mathfrak{p}}$ is either indecomposable or zero.
\item\label{thm1.2} A projective functor on $\mathcal{O}_0^{\mathfrak{p}}$ is completely determined, 
up to isomorphism, by the linear  transformation it induces on the level of the 
Grothendieck group. 
\end{enumerate}
\end{theorem}

For an explicit description of which indecomposable projective functors survive restriction 
from $\mathcal{O}_0$ to $\mathcal{O}_0^{\mathfrak{p}}$,  given in terms of 
Kazhdan-Lusztig combinatorics, we refer the reader to  Formula~\eqref{eq1}.

The approach to prove \cite[Theorem~5.1]{St} and \cite[Theorem~5.7]{St} in \cite{St} heavily relies
on the relation of  indecomposable projective functors which survive restriction from
$\mathcal{O}_0$ to $\mathcal{O}_0^{\mathfrak{p}}$, for maximal $\mathfrak{p}$, 
to braid avoiding permutations. This is, clearly,
not available for the general case. In fact, our approach to prove Theorem~\ref{thm1}\eqref{thm1.1} is 
completely different and is crucially based on several advances in the abstract $2$-representation theory of
finitary $2$-categories which were made in the series \cite{MM1,MM2,MM3,MM4,MM5,MM6} of papers
by Vanessa Miemietz and the second author. For Theorem~\ref{thm1}\eqref{thm1.2} we also use a
result of Steffen K{\"o}nig and Changchang Xi from \cite{KX} which asserts that the 
Cartan determinant of a cellular algebra is non-zero.

The paper is organized as follows. In Section~\ref{s1} we collect all necessary preliminaries from
the Lie algebra side of the story and then in Section~\ref{s2} we collect all necessary preliminaries from
the $2$-representation side.  Theorem~\ref{thm1} is proved in Section~\ref{s3}. The final Section~\ref{s4}
contains various speculations related to one of the original approaches to prove Theorem~\ref{thm1} which
did not work. This approach was based on an attempt to first prove the following:

\begin{conjecture}\label{conjecture2}
For $\mathfrak{g}=\mathfrak{sl}_n(\mathbb{C})$, let $L\in \mathcal{O}_0$ be simple and $\theta$ be
an indecomposable projective functor on $\mathcal{O}_0$. Then $\theta\, L$ is either an indecomposable module
or zero.
\end{conjecture}

In Section~\ref{s4} we discuss the evidence we have for the validity of this conjecture and also
approaches that could perhaps be used to prove it.
\vspace{5mm}

\noindent
{\bf Acknowledgment.} This research was done during the postdoctoral stay of the first author
at Uppsala University which was supported by the Knut and Alice Wallenbergs Stiftelse.
The second author is partially supported by the Swedish Research Council.
The second author thanks Catharina Stroppel for very interesting discussions during many 
failed joint attempts to prove Theorem~\ref{thm1} in the past. 
We thank Kevin Colembier for pointing out several typos in the original version of the paper.
Finally, we thank the referee for careful reading of the paper and very useful comments.
\vspace{5mm}

\section{Preliminaries from Lie theory}\label{s1}

\subsection{Generalities}\label{s1.1}

We work over $\mathbb{C}$. For a Lie algebra $\mathfrak{a}$,  we denote by $U(\mathfrak{a})$ 
the universal enveloping algebra of $\mathfrak{a}$.

\subsection{Category $\mathcal{O}$}\label{s1.2}

Let $\mathfrak{g}$ be a finite dimensional semi-simple complex Lie algebra 
with a fixed triangular decomposition $\mathfrak{g}= \mathfrak{n}_-\oplus\mathfrak{h}\oplus\mathfrak{n}_+$
and set $\mathfrak{b}:=\mathfrak{h}\oplus\mathfrak{n}_+$. With this decomposition one associates the corresponding 
{\em BGG category $\mathcal{O}$}, as defined in \cite{BGG}, which is the full subcategory of the
category of all $\mathfrak{g}$-modules consisting of all finitely generated  modules on which the
action of $\mathfrak{h}$ is diagonalizable and the action of $U(\mathfrak{n}_+)$ is locally finite.

For $\lambda\in \mathfrak{h}^*$, we denote by $L(\lambda)$ the {\em simple highest weight} $\mathfrak{g}$-module
with highest weight $\lambda$. These exhaust simple objects in $\mathcal{O}$, up to isomorphism.
The module $L(\lambda)$ is the unique simple quotient of the {\em Verma module} $\Delta(\lambda)$
and also of the {\em indecomposable projective module} $P(\lambda)$.
There is a contravariant simple preserving duality on $\mathcal{O}$ denoted by $M\mapsto M^{\star}$.
We set $\nabla(\lambda):=\Delta(\lambda)^{\star}$ and $I(\lambda):=P(\lambda)^{\star}$. Then 
$I(\lambda)$ is the {\em indecomposable injective envelope} of $L(\lambda)$.

We refer the reader to \cite{Hu} for more details on category $\mathcal{O}$.

\subsection{The principal block $\mathcal{O}_0$}\label{s1.3}

The Weyl group $W$ of $\mathfrak{g}$ acts on $\mathfrak{h}^*$ in the usual way.
We also consider the {\em dot-action} given by $w\cdot \lambda = w(\lambda + \rho) - \rho$ 
where $\rho$ is the half-sum of all positive roots. We denote by $w_0$ the longest element in $W$.

The {\em principal block} $\mathcal{O}_0$ of $\mathcal{O}$ is the Serre subcategory 
of $\mathcal{O}$ generated by all $L(w\cdot 0)$ for $w\in W$. It is a direct summand of 
$\mathcal{O}$ consisting of all modules in $\mathcal{O}$ which have the same generalized 
central character  as the trivial $\mathfrak{g}$-module $L(0)$.
To simplify notation, for $w\in W$, we set $L(w):= L(w\cdot 0)$, $\Delta(w):= \Delta(w\cdot 0)$,
$\nabla(w):= \nabla(w\cdot 0),$ $P(w):= P(w\cdot 0)$ and $I(w):= I(w\cdot 0)$. 

We consider the finite dimensional associative algebra
\begin{displaymath}
\mathtt{A} = \mathrm{End}_{\mathcal{O}}\big(\bigoplus_{w\in W}P(w)\big)^{\mathrm{op}}
\end{displaymath}
and have the usual equivalence between $\mathcal{O}_0$ and the category 
$\mathtt{A}\text{-}\mathrm{mod}$ of finite dimensional left $\mathtt{A}$-modules.
The algebra $\mathtt{A}$ is quasi-hereditary, with respect to the weight poset $W$ equipped
with the usual Bruhat order $\preceq$. The duality $\star$ on $\mathcal{O}$ restricts to  a duality on 
$\mathcal{O}_0$ which, in turn, gives an involution on $\mathtt{A}$ that fixes pointwise  
a complete set of primitive orthogonal idempotents, in particular,
$\mathtt{A}\cong \mathtt{A}^{\mathrm{op}}$.

\subsection{Parabolic subcategories}\label{s1.4}

Let $\mathfrak{p}$ be a parabolic subalgebra of $\mathfrak{g}$ containing $\mathfrak{b}$
and let $W_{\mathfrak{p}}$ be the corresponding {\em parabolic subgroup} of $W$.
The {\em parabolic category} $\mathcal{O}_0^{\mathfrak{p}}$ is the full subcategory of 
$\mathcal{O}_0$ consisting of all modules on which the action of $U(\mathfrak{p})$ is
locally finite. In can be alternatively described as the Serre subcategory of $\mathcal{O}_0$ 
generated by all $L(w)$ where $w\in W^{\mathfrak{p}}$, the set of {\em shortest} representatives of 
the cosets in ${}_{W_{\mathfrak{p}}}\hspace{-1mm}\setminus W$.

The exact inclusion of $\mathcal{O}_0^{\mathfrak{p}}$ into $\mathcal{O}_0$ admits a left adjoint,
denoted $Z^{\mathfrak{p}}$ and called the {\em Zuckerman functor}, defined as the largest quotient 
contained in  $\mathcal{O}_0^{\mathfrak{p}}$. For $w\in W^{\mathfrak{p}}$, 
we set $\Delta^{\mathfrak{p}}(w) := Z^{\mathfrak{p}}\Delta(w)$
and $P^{\mathfrak{p}}(w) := Z^{\mathfrak{p}}P(w)$. The duality $\star$ restricts to 
$\mathcal{O}_0^{\mathfrak{p}}$ and we set $\nabla^{\mathfrak{p}}(w) :=\Delta^{\mathfrak{p}}(w)^{\star}$
and $I^{\mathfrak{p}}(w):=P^{\mathfrak{p}}(w)^{\star}$.

The module $P^{\mathfrak{p}}(w)$ is the indecomposable projective cover of $L(w)$ in $\mathcal{O}_0^{\mathfrak{p}}$.
We set
\begin{displaymath}
\mathtt{A}^{\mathfrak{p}} = \mathrm{End}_{\mathcal{O}}
\big(\bigoplus_{w\in W^{\mathfrak{p}}}P^{\mathfrak{p}}(w)\big)^{\mathrm{op}}, 
\end{displaymath}
so $\mathcal{O}_0^{\mathfrak{p}}$ is equivalent to $\mathtt{A}^{\mathfrak{p}}$-mod. 
The algebra $\mathtt{A}^{\mathfrak{p}}$ is quasi-hereditary with respect to the restriction
of the Bruhat order to $W^{\mathfrak{p}}$ and also
$\mathtt{A}^{\mathfrak{p}}\cong(\mathtt{A}^{\mathfrak{p}})^{\mathrm{op}}$.

\subsection{Hecke algebra and Kazhdan-Lusztig combinatorics}\label{s1.5}

Denote by $S$ the set of simple reflections in $W$ corresponding to our fixed triangular
decomposition of $\mathfrak{g}$. Let $l:W\to \mathbb{Z}$ denote the length function on
$W$ with respect to $S$. Then, associated to the pair $(W,S)$, we have the Hecke 
algebra $\mathcal{H} = \mathcal{H}(W,S)$, which is a free $\mathbb{Z}[v,v^{-1}]$-module 
on generators $H_w$, where $w\in W$, and multiplication is uniquely 
defined using the following formulae: 
\begin{displaymath}
H_xH_y = H_{xy}\quad \text{ whenever }\quad l(xy) = l(x) + l(y)
\end{displaymath}
and
\begin{displaymath}
H_s^2= H_e + (v^{-1} - v)H_s, \quad \text{ for }\quad  s\in S.  
\end{displaymath}
We note that we use the normalization of \cite{So}.

There is a unique involution $\overline{\hspace{1mm}\cdot\hspace{1mm}}$ on $\mathcal{H}$ which maps 
$H_x\mapsto (H_{x^{-1}})^{-1}$  and $v\mapsto v^{-1}$, see \cite{KaLu,So}. We denote by $\underline{H}_w$, 
for $w\in W$,  the corresponding {\em Kazhdan-Lusztig basis} element, see \cite{So}.
Let $\leq_L$, $\leq_R$ and $\leq_J$ denote the {\em Kazhdan-Lusztig left, right} and {\em two-sided}
preorders, respectively. The corresponding equivalence relations will be denoted
$\sim_L$, $\sim_R$ and $\sim_J$, respectively, the are called {\em Kazhdan-Lusztig cells}. For
$w\in W$, we denote by $\mathcal{L}_w$, $\mathcal{R}_w$ and $\mathcal{J}_w$ the left,
right and two-sided Kazhdan-Lusztig cell containing $w$, respectively. In what follows we abbreviate, as usual, 
``Kazhdan-Lusztig'' simply by ``KL''.

In the special case of $\mathfrak{g} = \mathfrak{sl}_n$, we have $W = S_n$ where simple reflections 
are given by elementary transpositions. In this case there is a nice description of KL-cells 
in terms of the Robinson-Schensted correspondence (cf. \cite[Section~3.1]{Sa}), 
see for example \cite{Na} and references therein. 
In particular, this correspondence shows that different left (resp. right) KL-cells inside a 
two-sided KL-cell are not comparable with respect to the left (resp. right) preorders.
The same correspondence also shows that the  intersection of a right and a left KL-cells 
inside the same two-sided KL-cell consists of precisely one element. 
Finally, each left (or right) 
KL-cell contains a unique involution, called the {\em Duflo involution} 
of the KL-cell.

Lusztig's {\em $\mathbf{a}$-function} $\mathbf{a}:W\to\mathbb{Z}_{\geq 0}$, defined in 
\cite{Lu1} and \cite{Lu2}, is associated to the KL-combinatorics. 
In type $A$ it is uniquely determined by the properties that it is constant on two-sided KL-cells of $W$
and that $\mathbf{a}(w)=l(w)$ whenever $w$ is the longest element in $W_{\mathfrak{p}}$ for some
parabolic subalgebra $\mathfrak{p}$ of $\mathfrak{g}$ containing $\mathfrak{b}$.

\subsection{Subcategories associated to right KL-cells}\label{s1.6}

For a right KL-cell $\mathbf{R}$ of $W$, define $\hat{\mathbf{R}} := \{x\in W\mid x\leq_R \mathbf{R}\}$ 
and let $\mathcal{O}_0^{\hat{\mathbf{R}}}$ denote the Serre subcategory of $\mathcal{O}_0$ 
generated by all $L(x)$, where $x\in\hat{\mathbf{R}}$. These categories were introduced in \cite{MS}.
If $\mathfrak{p}$ is a parabolic subalgebra of $\mathfrak{g}$ containing $\mathfrak{b}$, and 
$w_0^{\mathfrak{p}}$ is the longest element in $W_{\mathfrak{p}}$, then $\mathcal{O}_0^{\mathfrak{p}}=\mathcal{O}_0^{\hat{\mathcal{R}}_{w_{\mathfrak{p}}}}$,
where $w_{\mathfrak{p}} = w_0^{\mathfrak{p}}w_0$, see \cite[Remark~14]{MS}.

Similarly to $\mathcal{O}_0^{\mathfrak{p}}$, the exact inclusion of $\mathcal{O}_0^{\hat{\mathbf{R}}}$ 
into $\mathcal{O}_0$ admits a left adjoint, denoted by $Z^{\hat{\mathbf{R}}}$. 
For $x\in \hat{\mathbf{R}}$, the indecomposable projective cover 
$Z^{\hat{\mathbf{R}}}P(x)$ of $L(x)$ in $\mathcal{O}_0^{\hat{\mathbf{R}}}$
is denoted by $P^{\hat{\mathbf{R}}}(x)$. We also define
\begin{displaymath}
\mathtt{A}^{\hat{\mathbf{R}}} := \mathrm{End}_{\mathcal{O}}
\big(\bigoplus_{x\in \hat{\mathbf{R}}}P^{\hat{\mathbf{R}}}(x)\big)^{\mathrm{op}},
\end{displaymath}
so that $\mathcal{O}_0^{\hat{\mathbf{R}}}$ is equivalent to $\mathtt{A}^{\hat{\mathbf{R}}}$-mod.
The duality $\star$ restricts to $\mathcal{O}_0^{\hat{\mathbf{R}}}$ and hence gives 
an involution on $\mathtt{A}^{\hat{\mathbf{R}}}$ which stabilizes a fixed set of primitive orthogonal idempotents.
We note that $\mathtt{A}^{\hat{\mathbf{R}}}$ is not quasi-hereditary in general, see \cite[Section~5.3]{MS}.

For $x\in \hat{\mathbf{R}}$, the indecomposable projective module $P^{\hat{\mathbf{R}}}(x)$
in $\mathcal{O}_0^{\hat{\mathbf{R}}}$ is injective if and only if $x\in \mathbf{R}$,
see \cite[Theorem~6]{Ma}.

\subsection{Graded setup}\label{s1.7}

By graded, we will always mean $\mathbb{Z}$-graded.

A graded finite dimensional associative algebra $\displaystyle A=\bigoplus_{i\in\mathbb{Z}}A_i$
is called {\em positively graded} provided that $A_i=0$ for $i<0$ and $A_0$ is semi-simple.
Consider the category $A$-gmod of finite dimensional graded $A$-modules with the usual
endofunctor $\langle 1\rangle$ which shifts the grading. For a graded $A$-module $M$,
we write
\begin{displaymath}
\max(M) = \max\{i\in \mathbb{Z}\mid M_i\neq 0\}\quad\text{ and }\quad
\min(M) = \min\{i\in\mathbb{Z}\mid M_i\neq 0\}. 
\end{displaymath}
The {\em graded length} of $M$ is then $\mathrm{grl}(M) := \max(M) -\min(M)+1$.

If $A$ is positively graded and $M$ is a graded $A$-module, 
then $M$ has a filtration by $A$-submodules 
$0 = M^{m+1}\subseteq M^{m}\subseteq \cdots \subseteq M^{k+1}\subseteq M^k = M$ 
where we have $k = \min(M)$, $m=\max(M)$ and $\displaystyle M^j = \bigoplus_{i\geq j}M_i$. 
This filtration is called the {\em grading filtration}. All subquotients of this filtration are
semi-simple.

\subsection{Graded category $\mathcal{O}$}\label{s1.8}

The algebra $\mathtt{A}$ can be positively graded. In fact, it is a Koszul algebra,
see \cite{So1,BGS,Ma0}. The algebras $\mathtt{A}^{\mathfrak{p}}$ 
and $\mathtt{A}^{\hat{\mathbf{R}}}$ are quotients of $\mathtt{A}$ and in this way they inherit
from $\mathtt{A}$ a positive grading, see for example \cite[Section~2.3]{Ma1}. We denote by 
$\mathtt{A}_{\mathbb{Z}}$, $\mathtt{A}^{\mathfrak{p}}_{\mathbb{Z}}$ and 
$\mathtt{A}^{\hat{\mathbf{R}}}_{\mathbb{Z}}$ the graded versions of these algebras, 
and by ${}^{\mathbb{Z}}\mathcal{O}_0$, ${}^{\mathbb{Z}}\mathcal{O}_0^{\mathfrak{p}}$ 
and ${}^{\mathbb{Z}}\mathcal{O}_0^{\hat{\mathbf{R}}}$ the corresponding categories of finite dimensional
graded modules.

For each of the above categories, we then have a forgetful functor to the corresponding 
non-graded categories $\mathcal{O}_0$, $\mathcal{O}_0^{\mathfrak{p}}$ and $\mathcal{O}_0^{\hat{\mathbf{R}}}$. 
We call a module $M$ in any of the non-graded categories {\em gradable} if it is isomorphic to
an image of some graded module $\widetilde{M}$ under the corresponding forgetful functor.
The module  $\widetilde{M}$ is then called a {\em graded lift} of $M$. 
All the modules $L(x)$, $P(x)$, $P^{\mathfrak{p}}(x)$ and $P^{\hat{\mathbf{R}}}(x)$, for $x\in W$, are 
gradable, see \cite[Theorem~2.1]{St1} and \cite[Theorem~2.1]{St}. By \cite[Lemma 1.5]{St1}, 
a graded lift of any of the aforementioned modules will be unique up to isomorphism and 
shift of grading. If $M$ is one of those modules, we will denote by $\widetilde{M}$ the graded 
lift whose head is concentrated in degree zero. Similarly, we define graded lifts 
for Verma modules and their quotients. For dual Verma modules and injective modules, 
the standard graded lift is the one in which 
the socle is concentrated in degree zero.

\subsection{Projective functors}\label{s1.9}

Following \cite{BG}, a {\em projective functor} on $\mathcal{O}$  is a functor isomorphic to a direct 
summand of the functor of tensoring with a finite dimensional $\mathfrak{g}$-module. Projective functors 
on $\mathcal{O}_0$ were classified in \cite{BG}. It turns out that indecomposable projective functors
on  $\mathcal{O}_0$ are in bijection with elements in $W$. For $w\in W$, we denote the corresponding
projective endofunctor of $\mathcal{O}_0$ by $\theta_w$. The functor $\theta_w$ is the unique, up
to isomorphism, projective functor with the property $\theta_w P(e) = P(w)$.

From \cite[Theorem~8.2]{St1} (see also \cite[Corollary~3.2]{St}) it follows that each 
$\theta_w$ admits a graded lift to an endofunctor of ${}^{\mathbb{Z}}\mathcal{O}_0$ and 
this lift is unique up to isomorphism and shift of grading.  We denote the corresponding functor by
$\widetilde{\theta}_w$ which we normalize by the condition
$\widetilde{\theta}_w\widetilde{P}(e) = \widetilde{P}(w)$.

\subsection{Decategorification of the action of projective functors}\label{s1.10}

Let $\tilde{\mathcal{P}}$ denote the additive tensor category of graded projective functors. 
It acts on ${}^{\mathbb{Z}}\mathcal{O}_0$ in the natural way. The Grothendieck group 
$[{}^{\mathbb{Z}}\mathcal{O}_0]$
of ${}^{\mathbb{Z}}\mathcal{O}_0$
is identified with $\mathcal{H}$ by sending $[\widetilde{\Delta}(w)]$ to $H_w$
under the convention that multiplication by $v$ corresponds to the shift $\langle 1\rangle$ of grading.
The split Grothendieck group $[\tilde{\mathcal{P}}]_{\oplus}$
of $\tilde{\mathcal{P}}$ is similarly identified with 
$\mathcal{H}$ by sending $[\widetilde{\theta}_w]$ to $\underline{H}_w$.
In this way the action of $\tilde{\mathcal{P}}$ on ${}^{\mathbb{Z}}\mathcal{O}_0$ gives the
right regular representation of $\mathcal{H}$, see \cite[Theorem~7.11]{Ma} for more details.
The ungraded version ${\mathcal{P}}$ of $\tilde{\mathcal{P}}$ is defined similarly and categorifies the
right regular $\mathbb{Z}[W]$-module, see also Subsection~\ref{s2.5}.

\section{Preliminaries from $2$-representation theory}\label{s2}

\subsection{Finitary and fiat $2$-categories}\label{s2.1}

We denote by $\mathbf{Cat}$ the category of small categories. A $2$-category is a category 
{\em enriched over} $\mathbf{Cat}$. Thus, a $2$-category $\cC$ consists of objects and 
morphism categories $\cC(\mathtt{i},\mathtt{j})$, objects of which, in turn, are $1$-morphisms of
$\cC$ and morphisms of which are $2$-morphisms of $\cC$. As usual, we denote by $\circ_0$ 
and $\circ_1$ the horizontal and vertical composition of $2$-morphisms, respectively. We refer
to \cite{Le} for more details.

Following \cite{MM1}, we call a $2$-category $\cC$ {\em finitary} provided that
\begin{itemize}
\item $\cC$ has finitely many objects;
\item each morphism category $\cC(\mathtt{i},\mathtt{j})$ is $\mathbb{C}$-linear, additive and
idempotent split with finitely many isomorphism classes of indecomposable objects and 
finite dimensional spaces of $2$-morphisms;
\item all compositions are biadditive and $\mathbb{C}$-bilinear when applicable;
\item all identity $1$-morphisms are indecomposable.
\end{itemize}
Furthermore, we call a finitary $2$-category $\cC$ {\em weakly fiat} provided that
\begin{itemize}
\item $\cC$ has a weak anti-automorphism $(-)^*$ which reverses direction of both $1$-morphisms and
$2$-morphisms;
\item $\cC$ has {\em adjunction $2$-morphisms} $\alpha: F\circ F^*\rightarrow \mathbbm{1}_\mathtt{j}$ and 
$\beta: \mathbbm{1}_\mathtt{i}\rightarrow F^*\circ F$ such that 
$\alpha_F\circ_1 F(\beta) = \mathrm{id}_F$ and $F^*(\alpha)\circ_1\beta_{F^*} = \mathrm{id}_{F^*}$. 
\end{itemize}
Here $\mathbbm{1}_\mathtt{j}$ is the identity $1$-morphism for the object  $\mathtt{j}$, further, 
$\mathrm{id}_F$ is the identity $2$-morphism for the $1$-morphism $F$ and, finally,
$F(\beta)$ stands for $\mathrm{id}_F\circ_0 \beta$ and $\alpha_F$ stands for $\alpha\circ_0\mathrm{id}_F$.
If $(-)^*$ is involutive, then $\cC$ is called {\em fiat}. For example, $\mathcal{P}$ is biequivalent to
a fiat $2$-category, see Subsection~\ref{s2.5} for details.

A {\em $2$-representation} of a $2$-category $\cC$ is a strict $2$-functor from $\cC$ to
$\mathbf{Cat}$. All $2$-representations of $\cC$ form a $2$-category, denoted $\cC$-mod,
in which $1$-morphisms are non-strict $2$-natural transformations and $2$-morphisms are modifications,
see e.g. \cite{MM3} for details.

A {\em $2$-representation} of $\cC$ is called {\em additive} if it is given by an additive 
$\mathbb{C}$-linear action of $\cC$ on additive, idempotent split, $\mathbb{C}$-linear categories
with finitely many isomorphism classes of indecomposable objects. The $2$-category of 
additive $2$-representations of $\cC$ is denoted by $\cC$-amod.

\subsection{Combinatorics of finitary $2$-categories}\label{s2.2}

For a finitary $2$-category $\cC$ consider the set $\mathcal{S}_{\ccC}$ of isomorphism classes of 
indecomposable $1$-morphisms in $\cC$. The set $\mathcal{S}_{\ccC}$ has the natural structure of a
{\em multisemigroup} given by
\begin{displaymath}
[F]\circ [G]=\{[H]\,\vert\, H\text{ is isomorphic to a direct summand of } F\circ G\},
\end{displaymath}
see \cite[Section~3]{MM2} for details. The left preorder $\leq_L$ on $\mathcal{S}_{\ccC}$ is given by 
$F\leq_L G$ if and only if $[G]\in \mathcal{S}_{\ccC}\circ [F]$. An equivalence class of 
$\leq_L$ is called a {\em left cell}. The right preorder $\leq_R$ and the corresponding right
cells are defined similarly using right multiplication. The two-sided preorder $\leq_J$ 
and the corresponding two-sided cells are defined similarly using two-sided multiplication.

\subsection{Principal and cell $2$-representations}\label{s2.3}

For a finitary $2$-category 
$\cC$ and an object $\mathtt{i}\in\cC$, we denote by $\mathbf{P}_{\mathtt{i}}$
the corresponding {\em principal} $2$-representation $\cC(\mathtt{i},{}_-)$. 

Let $\mathcal{L}$ be a left cell in $\mathcal{S}_{\ccC}$ and $\mathtt{i}$ be the object in $\cC$
which is the origin of all $1$-morphisms in $\mathcal{L}$. Denote by $\mathbf{N}_{\mathcal{L}}$
the additive closure inside $\mathbf{P}_{\mathtt{i}}$ of all $1$-morphisms $F$ such that $F\geq_L\mathcal{L}$.
Then $\mathbf{N}_{\mathcal{L}}$ is a $2$-representation of $\mathcal{S}_{\ccC}$ by restriction.
This $2$-representation has a unique maximal $\cC$-invariant ideal $\mathbf{I}_{\mathcal{L}}$
and the quotient $\mathbf{N}_{\mathcal{L}}/\mathbf{I}_{\mathcal{L}}$ is called the 
{\em cell $2$-representation} of $\cC$ corresponding to $\mathcal{L}$ and denoted by $\mathbf{C}_{\mathcal{L}}$,
see \cite[Section~6.5]{MM2} for details.

A two-sided cell $\mathcal{J}$ is called {\em regular} provided that different left (resp. right) 
cells inside $\mathcal{J}$ are not comparable with respect to the left (resp. right) order.
A regular two-sided cell $\mathcal{J}$ is called {\em strongly regular} provided that 
the intersection of any left and any right cell inside $\mathcal{J}$ consists of exactly one element.
For example, all two-sided cells of the tensor category $\mathcal{P}$
for $\mathfrak{g}$ of type $A$ are strongly regular, see \cite[Subsection~7.1]{MM1} for details.
We note that, due to the fact that the action of $\mathcal{P}$ on $\mathcal{O}_0$ is a {\em right}
action, the Kazhdan-Lusztig left (resp. right) order as defined in Subsection~\ref{s1.5} corresponds to the
right (resp. left) order as defined in this subsection.

\subsection{Transitive and simple transitive $2$-representations}\label{s2.4}

Let $\cC$ be a finitary $2$-category and $\mathbf{M}\in\cC$-amod. The $2$-representation $\mathbf{M}$
is called {\em transitive} provided that for any $\mathtt{i}\in\cC$ and
any indecomposable object $X\in \mathbf{M}(\mathtt{i})$, the additive closure of all objects of
the form $\mathbf{M}(F)\, X$, where $F$ is a $1$-morphism in $\cC$, equals $\mathbf{M}$.
A transitive $2$-representation $\mathbf{M}$ of $\cC$ is called {\em simple transitive} if it
does not have any proper $\cC$-invariant ideals, see \cite{MM5} for details. The arguments in this
paper use crucially the following statement proved in \cite[Theorem~43]{MM1},
\cite[Theorem~18]{MM5} and \cite[Theorem~31]{MM6}.

\begin{theorem}\label{thm3}
Let $\cC$ be a weakly fiat $2$-category in which all two-sided cells are strongly regular.
\begin{enumerate}[$($i$)$]
\item\label{thm3.1} Each simple transitive $2$-representation of $\cC$ is equivalent to a cell $2$-representation.
\item\label{thm3.2} For any two left cells inside the same two-sided cell, the corresponding 
cell $2$-representation of $\cC$ are equivalent.
\end{enumerate}
\end{theorem}

\subsection{The $2$-category of projective functors}\label{s2.5}

We fix some small category $\mathcal{A}$ equivalent to $\mathcal{O}_0$ and consider the
$2$-category $\cP$ of {\em projective endofunctors} of $\mathcal{A}$, see \cite[Subsection~7.1]{MM1}.
This is a fiat $2$-category with one object $\mathtt{i}$ (which should be thought of as $\mathcal{A}$)
and a weak involution given by $\theta_w^*=\theta_{w^{-1}}$, for $w\in W$.
Indecomposable $1$-morphisms in this category are exactly $\theta_w$, for $w\in W$, up to isomorphism.
Because of the right action, cell $2$-representations of $\cP$ are indexed by {\em right} 
KL-cells. For a right KL-cell $\mathbf{R}$, the corresponding cell $2$-representation of
$\cP$ is equivalent to the action of $\cP$ on the additive category of projective-injective modules
in $\mathcal{O}_0^{\hat{\mathbf{R}}}$, see \cite{MS} and \cite{MM1} for details.

If $\mathfrak{g}$ is of type $A$, then, as already mentioned in Subsection~\ref{s1.5}, all
two-sided cells  in $\cP$ are strongly regular. In particular, 
Theorem~\ref{thm3} gives a complete description of all simple transitive $2$-representations of 
$\cP$ in this case, up to equivalence.

Recall (see for example \cite[Lemma~12]{MM1} or \cite[Proposition~1]{Ma1}
or \cite[(1.4)]{mathas}) that, for $x,y\in W$, we have
\begin{equation}\label{eq1}
\theta_x L(y)\neq 0 \quad\text{ if and only if }\quad x^{-1}\leq_L y
\quad\text{ if and only if }\quad x\leq_R y^{-1}.
\end{equation}
For any KL-right cell $\mathbf{R}$, we have  
$L(y)\in \mathcal{O}_0^{\hat{\mathbf{R}}}$ if and only if $y\leq_R \mathbf{R}$.
This, combined with Formula~\eqref{eq1}, completely determines which $\theta_x$ survive 
restriction to $\mathcal{O}_0^{\hat{\mathbf{R}}}$.

Now, for a fixed parabolic $\mathfrak{p}$ in $\mathfrak{g}$ containing $\mathfrak{b}$,
we can consider the Serre subcategory $\mathcal{A}^{\mathfrak{p}}$ of $\mathcal{A}$
which corresponds to $\mathcal{O}_0^{\mathfrak{p}}$. The action of $\cP$ preserves
$\mathcal{O}_0^{\mathfrak{p}}$ and hence we can define the $2$-category $\cP^{\mathfrak{p}}$
as the $2$-category given by the additive closure of the $2$-action of $\cP$ on 
$\mathcal{O}_0^{\mathfrak{p}}$, the so-called {\em image completion} of the $2$-action in the
sense of \cite[Subsection~7.3]{MM2}. In more detail:
\begin{itemize}
\item the $2$-category $\cP^{\mathfrak{p}}$ has the same objects as $\cP$;
\item $1$-morphisms in $\cP^{\mathfrak{p}}$ are all 
endofunctors of $\mathcal{O}_0^{\mathfrak{p}}$ which belong to the additive closure of 
endofunctors given by the action of $1$-morphisms in $\cP$ on $\mathcal{O}_0^{\mathfrak{p}}$;
\item $2$-morphisms in $\cP^{\mathfrak{p}}$ are all natural transformations of 
endofunctors of $\mathcal{O}_0^{\mathfrak{p}}$.
\end{itemize}
We note that, while $\cP$ is fiat, the
$2$-category $\cP^{\mathfrak{p}}$ is, at the present stage, only weakly fiat.
In fact, it is Theorem~\ref{thm1}\eqref{thm1.1} which implies that $\cP^{\mathfrak{p}}$ is also fiat.

We would like to mention once more that, due to the fact that the action of $\mathcal{P}$ 
on $\mathcal{O}_0$ is a {\em right} action, the Kazhdan-Lusztig left (resp. right) order 
as defined in Subsection~\ref{s1.5} corresponds to the right (resp. left) order as defined in 
this subsection. In particular, for any simple reflection $s$ and any $w\in W$ such that $l(sw)>l(w)$,
we have $\theta_s\theta_w= \theta_{ws}\oplus \text{other terms}$, see \cite{BG,St,MS} for details.

\section{Proof of Theorem~\ref{thm1}}\label{s3}

\subsection{Proof of Theorem~\ref{thm1}\eqref{thm1.1}}\label{s3.1}

For the rest of this section we set $\mathfrak{g} = \mathfrak{sl}_n$ and let $\mathfrak{p}$ be a 
parabolic subalgebra of $\mathfrak{g}$ containing $\mathfrak{b}$. 
Recall that $w_0^{\mathfrak{p}}$ denotes the longest element of the 
parabolic Weyl group $W_{\mathfrak{p}}$ corresponding to $\mathfrak{p}$. 
Set $w_{\mathfrak{p}} = w_0^{\mathfrak{p}}w_0$ and let 
$\mathbf{R}:=\mathcal{R}_{w_{\mathfrak{p}}}$. Then we have 
$\mathcal{O}_0^{\mathfrak{p}} = \mathcal{O}_0^{\hat{\mathbf{R}}}$.

For a module $M$, denote by $\ell\ell(M)$ the Loewy length of $M$, i.e. 
the shortest length of a filtration of $M$ with semi-simple quotients.

\begin{lemma}\label{glll}
Let $M\in {}^{\mathbb{Z}}\mathcal{O}_0^{\mathfrak{p}}$. Then $\ell\ell(M)\leq \mathrm{grl}(M)$.
\end{lemma}

\begin{proof}
Since $\mathtt{A}_{\mathbb{Z}}$ is Koszul, it is positively graded. 
This property is inherited by the quotient $\mathtt{A}^{\mathfrak{p}}_{\mathbb{Z}}$.
Thus, the grading filtration of $M$ is a filtration of length $\mathrm{grl}(M)$ with semi-simple 
subquotients, and thus $\ell\ell(M)\leq \mathrm{grl}(M)$.
\end{proof}

Let, from now on, $x\in W$ be such that $\theta_x$ is non-zero when restricted to $\mathcal{O}_0^{\mathfrak{p}}$.

\begin{lemma} \label{intersectstwosided}
There is some $y\in \hat{\mathbf{R}}$ such that $x\sim_{J} y$.
\end{lemma}

\begin{proof} 
Since $\theta_x$ is non-zero when restricted to $\mathcal{O}_0^{\hat{\mathbf{R}}}$, 
there is some $z\leq_R w_{\mathfrak{p}}$ such that $x\leq_R z^{-1}$, see Formula~\eqref{eq1}. 
But then $x\leq_R z^{-1}\sim_{J} z\leq_R w_{\mathfrak{p}}$ and thus $x\leq_{J} w_{\mathfrak{p}}$.
Therefore we have to show that for any two-sided KL-cell $\mathbf{J}$ such that 
$\mathbf{J}\leq_{J} \mathbf{R}$ we have  $\mathbf{J}\cap \hat{\mathbf{R}}\neq\emptyset$.

To prove this we recall that the action of projective functors on $\mathcal{O}_0^{\mathfrak{p}}$ 
categorifies, after extending scalars to $\mathbb{C}$, the induced sign 
$\mathbb{C}[W]$-module by \cite[Proposition~30]{MS}.
This module is a direct sum of Specht modules, where the Specht module for a partition 
$\lambda$ occurs at least once whenever $\lambda\unlhd\mu$, where $\mu$ is the partition 
corresponding to $\mathfrak{p}$ and $\unlhd$ denotes the dominance ordering, see \cite[Corollary~2.4.7]{Sa}. 
On the other hand, the Kazhdan-Lusztig cell $\mathbb{C}[W]$-module  
associated to a right KL-cell inside a two-sided KL-cell is exactly
the Specht module for the partition corresponding to the two-sided KL-cell via the 
Robinson-Schensted correspondence, see \cite[Theorem~1.4]{KaLu} and \cite[Theorem~4.1]{Na}. 
The $2$-representation corresponding to the action of projective functors on 
$\mathcal{O}_0^{\mathfrak{p}}$ has a weak Jordan-H{\"o}lder series in the sense of  
\cite[Section~4.3]{MM5} corresponding to  the right KL-cells in $\hat{\mathbf{R}}$
(the corresponding subquotients are unique in the sense of \cite[Theorem 8]{MM5}).
In the Grothendieck group, this gives a Jordan-H{\"o}lder series for the induced sign $\mathbb{C}[W]$-module
by KL-cell modules corresponding to the right KL-cells which appear in $\hat{\mathbf{R}}$.
Hence, whenever a Specht module corresponding to a partition $\lambda$ occurs 
in the induced sign $\mathbb{C}[W]$-module, there must be a two-sided KL-cell corresponding to 
$\lambda$ which intersects $\hat{\mathbf{R}}$ non-trivially. Since the dominance order coincides 
with the two-sided order by \cite[Theorem~5.1]{Ge}, the claim of the lemma follows.
\end{proof}

Let $\mathbf{J}$ be the two-sided cell containing $x$ and 
$\mathbf{R}'\subseteq \hat{\mathbf{R}}$ be a right cell such that 
$\mathbf{R}'\cap \mathbf{J}\neq\emptyset$, which exists by Lemma \ref{intersectstwosided}
(note that $\mathbf{R}'$ is not uniquely determined by these properties).

\begin{lemma}\label{uniquesummand}
There is a unique $y\in \mathbf{R}'$ such that $\theta_x L(y)\neq 0$ 
and with this choice of $y$ the module $\theta_x L(y)$ is indecomposable.
\end{lemma}

\begin{proof}
According to Formula~\eqref{eq1}, the inequality $\theta_x L(y)\neq 0$ is equivalent to 
the inequality $x^{-1}\leq_L y$.
Since $x\sim_J y$ by assumptions, we have $x^{-1}\sim_J y$. Together with $x^{-1}\leq_L y$,
we thus have $x^{-1}\sim_L y$ by regularity of $\mathbf{J}$. Due to strong
regularity of $\mathbf{J}$, we thus have that $y$
is the unique element in $\mathcal{L}_{x^{-1}}\cap\mathbf{R}'$.

Now let $y$ be given as above and let $\mathbf{R}'' = \mathcal{R}_{x}$. 
By \cite[Theorem~43]{MM1} and \cite[Subsection~7.1]{MM1}, 
the cell $2$-representations of $\cP$ corresponding 
to $\mathbf{R}'$ and $\mathbf{R}''$ are equivalent, 
so it suffices to prove that $\theta_x L(y')$ is indecomposable if we take 
$y'\in \mathcal{L}_{x^{-1}}\cap \mathbf{R}''$. However, the unique element in this latter intersection 
is precisely the Duflo involution in $\mathbf{R}''$, and then 
$\theta_x L(y') = P^{\hat{\mathbf{R}}''}(x)$, see \cite[Theorem~6]{Ma1} or \cite[Section~4.5]{MM1}, 
which is indecomposable.
\end{proof}

Apart from the above, we will also need the following lemma.

\begin{lemma}\label{remainstransitive}
Let $\cC$ be a finitary $2$-category and $\cC'$ be an image completion of $\cC$.
Let $\Psi: \cC\to\cC'$ be the corresponding canonical $2$-functor.
Then the pullback, via $\Psi$, of a transitive $2$-representation of $\cC'$
is a transitive $2$-representation of $\cC$.
\end{lemma}

\begin{proof}
This is clear from the definitions since, for any $1$-morphism $F$ in $\cC'$,
there is a $1$-morphism $G$ in $\cC'$ and a $1$-morphism $H$ in $\cC$
such that $F\oplus G$ is isomorphic to $\Psi(H)$.
\end{proof}

Consider $x\in W$ such that the restriction of ${\theta}_x$ to $\mathcal{O}_0^{\mathfrak{p}}$ is non-zero.
Assume that this restriction decomposes. Let $\overline{\theta}_x$ denote the 
unique indecomposable direct summand of $\theta_x$ such that $\overline{\theta}_x L(y)\neq 0$ 
where $y$ is as in Lemma~\ref{uniquesummand}. Let 
$F_x$ be such that $\theta_x = \overline{\theta}_x \oplus F_x$.

Assume that $F_x$ is non-zero and consider some $z\in W$ such that $F_x L(z)\neq 0$.
Choose $z$ in a two-sided cell such that $\mathbf{a}(z)$ is minimal possible with the property $F_x L(z)\neq 0$. 
Because of Lemma~\ref{uniquesummand} and also our choice of 
$\overline{\theta}_x$, we cannot have $x \sim_R z^{-1}$, so we have $x <_R z^{-1}$
and hence $x<_J z$. Thus, by \cite[Proposition~1]{Ma1}, we have 
$\max({\theta}_x L(z)) < \mathbf{a}(z)$ and $\min({\theta}_x L(z)) > -\mathbf{a}(z)$,
implying the inequality $\ell\ell({F}_x L(z))\leq \ell\ell(\theta_x L(z))< 2\mathbf{a}(z)+1$
by Lemma~\ref{glll} (see Subsection~\ref{s1.7} for definitions of $\min$ and $\max$).

Consider the defining $2$-representation $\mathbf{N}$ of $\cP^{\mathfrak{p}}$. 
Let $\mathbf{M}$ be the induced additive $2$-representation of $\cP^{\mathfrak{p}}$
on the additive closure of all objects of the form $\theta_w L(z)$, where $w\in W$.
Let $\cP_{\mathfrak{p}}$ be the quotient of $\cP$ by the $2$-ideal generated by all $\theta_w$
which annihilate $\mathcal{O}_0^{\mathfrak{p}}$. Then $\cP_{\mathfrak{p}}$ is fiat and
all two-sided cells in $\cP_{\mathfrak{p}}$ are strongly regular, by construction. 
Moreover, by Lemma~\ref{intersectstwosided}, the indexing set for elements of
each two-sided cell of $\cP_{\mathfrak{p}}$ intersects the set $\hat{\mathbf{R}}$.

Now, let $\mathbf{M}'$ be the simple transitive subquotient of $\mathbf{M}$ containing ${F}_x L(z)$.
So far, this is a $2$-representation of $\cP^{\mathfrak{p}}$.
We may consider $\mathbf{M}'$ as a $2$-representation of $\cP_{\mathfrak{p}}$ via the canonical $2$-functor 
$\cP_{\mathfrak{p}}\hookrightarrow \cP^{\mathfrak{p}}$. This is, by construction, a transitive $2$-representation 
of $\cP^{\mathfrak{p}}$ and hence also of $\cP_{\mathfrak{p}}$, by 
Lemma~\ref{remainstransitive}. Let $\underline{\mathbf{M}'}$ be the simple 
transitive quotient of $\mathbf{M}'$, now as a $2$-representation of $\cP_{\mathfrak{p}}$.
From the above estimates and construction, we have the inequalities 
$0< \ell\ell({F}_x L(z)) < 2\mathbf{a}(z)+1$.

By Theorem~\ref{thm3}\eqref{thm3.1}, $\underline{\mathbf{M}'}$ is equivalent 
to a cell $2$-representation of $\cP_{\mathfrak{p}}$ which corresponds to some right KL-cell,
say $\mathcal{R}\subset \hat{\mathbf{R}}$. Since $\ell\ell({F}_x L(z)) < 2\mathbf{a}(z) + 1$, 
we cannot have $\mathbf{a}(\mathcal{R})\geq \mathbf{a}(z)$. Indeed, 
the cell $2$-representation corresponding to $\mathcal{R}$ consists of objects 
of Loewy length $2\mathbf{a}(\mathcal{R}) + 1$ by \cite[Corollary 7]{Ma1}
since this cell $2$-representation is modelled on the category of projective-injective
objects corresponding to $\mathcal{R}$ and they all have Loewy length $2\mathbf{a}(\mathcal{R}) + 1$. 
Therefore this cell $2$-representation cannot contain the object ${F}_x L(z)$ of strictly smaller Loewy length.

On the other hand, by construction,  $\underline{\mathbf{M}'}$ does not annihilate $F_x$
and $z$ is chosen such that, for all $z'\in \hat{\mathbf{R}}$ 
with $\mathbf{a}(z')<\mathbf{a}(z)$, we have $F_x L(z')=0$.
Hence $\mathbf{a}(\mathcal{R})<\mathbf{a}(z)$ is not possible either.
The obtained contradiction shows that $F_x=0$ and completes the proof of Theorem~\ref{thm1}\eqref{thm1.1}.

\subsection{Proof of Theorem~\ref{thm1}\eqref{thm1.2}}\label{s3.2}

Let $F$ and $G$ be two projective functors  on $\mathcal{O}_0^{\mathfrak{p}}$. 
By Theorem~\ref{thm1}\eqref{thm1.1}, we may write
\begin{displaymath}
F \cong \bigoplus_{w\leq_J w_{\mathfrak{p}}}a_w\theta_w\quad\text{ and }
\quad G \cong \bigoplus_{w\leq_J w_{\mathfrak{p}}}b_w\theta_w 
\end{displaymath}
for some non-negative integers $a_w$ and $b_w$. We need to show that, 
if $F$ and $G$ induce the same linear operators on the Grothendieck group $[\mathcal{O}_0^{\mathfrak{p}}]$ of 
$\mathcal{O}_0^{\mathfrak{p}}$, then $a_w = b_w$ for all $w\in W$.

Using Formula~\eqref{eq1} and induction on the two-sided order, 
it is sufficient to consider the case 
\begin{displaymath}
F \cong \bigoplus_{w\in \mathbf{J}}a_w\theta_w\quad\text{ and }
\quad G \cong \bigoplus_{w\in \mathbf{J}}b_w\theta_w, 
\end{displaymath}
where $\mathbf{J}$ is a fixed two-sided KL-cell such that $\mathbf{J}\leq_J w_{\mathfrak{p}}$.
Let $\mathbf{R}'$ be a right KL-cell in $\mathbf{J}\cap\hat{\mathbf{R}}$, which exists
due to Lemma~\ref{intersectstwosided}.

As the classes of simple modules $L(x)$, for $x\in \mathbf{R}'$, are linearly independent in 
$[\mathcal{O}_0^{\mathfrak{p}}]$, it suffices to show that, for $w\in \mathbf{J}$, the matrices
\begin{displaymath}
M_w:=\big([\theta_w L(x):L(y)]\big)_{y,x\in \mathbf{R}'}
\end{displaymath}
are linearly independent. Since $\mathbf{J}$ is strongly regular, using Formula~\eqref{eq1} we see that it is
enough to show that the matrices $M_w$ are linearly independent for $w$ in a fixed left KL-cell
$\mathbf{L}$ of $\mathbf{J}$. Note that, by Formula~\eqref{eq1}, each $M_w$ has a unique non-zero column
(our convention is that columns are indexed by $x$).

Consider the cell $2$-representation $\mathbf{C}_{\mathbf{R}'}$ of $\cP$. Let $\mathtt{Q}_{\mathbf{R}'}$
be the opposite of the endomorphism algebra of the multiplicity free sum of all indecomposable
projective-injective modules in $\mathcal{O}_0^{\hat{\mathbf{R}}'}$. Then 
$\mathbf{C}_{\mathbf{R}'}(\mathtt{i})\cong \mathtt{Q}_{\mathbf{R}'}$-mod.
By \cite[Theorem~43]{MM1}, the functors $\theta_w$, for $w\in\mathbf{L}$, act as projective
functors on $\mathtt{Q}_{\mathbf{R}'}$-mod in the sense of \cite[Section~7.3]{MM1}.
Hence, putting together the non-zero columns of the matrices $M_w$, for
$w\in \mathbf{L}$, produces the Cartan matrix of $\mathtt{Q}_{\mathbf{R}'}$. Therefore we only need
to show that the Cartan matrix of $\mathtt{Q}_{\mathbf{R}'}$ is non-degenerate.

In fact, we claim that $\mathtt{Q}_{\mathbf{R}'}$ is a cellular algebra, in which case the fact that 
its Cartan matrix is non-degenerate follows from \cite[Proposition~1.2]{KX}

To prove that $\mathtt{Q}_{\mathbf{R}'}$ is a cellular algebra, consider another right cell
$\mathbf{R}''$ in $\mathbf{J}$ which we choose such that $\mathbf{R}''$ contains
$w_{\mathfrak{q}}$ for some parabolic subalgebra $\mathfrak{q}$ of $\mathfrak{g}$ containing $\mathfrak{b}$.
This is possible because we are in type $A$. Consider $\mathcal{O}_0^{\mathfrak{q}}$
and let $T$ be a multiplicity free direct sum of all indecomposable projective-injective
objects in $\mathcal{O}_0^{\mathfrak{q}}$. Then the opposite of the endomorphism algebra of $T$ is
isomorphic to $\mathtt{Q}_{\mathbf{R}'}$ by \cite[Theorem~5.4]{MS0}
(see also \cite[Theorem~18]{MS} or \cite[Theorem~43]{MM1}).
On the other hand, the associative algebra $\mathtt{A}^{\mathbf{R}''}$ of 
$\mathcal{O}_0^{\mathfrak{q}}$ is quasi-hereditary with simple preserving duality.  
In particular,  $\mathtt{A}^{\mathbf{R}''}$ 
is cellular by  \cite[Corollary~4.2]{KX0}. As the duality fixes projective-injective modules,
it fixes $T$ and hence the endomorphism algebra of $T$ is cellular by \cite[Proposition~4.3]{KX0}.
This completes the proof of Theorem~\ref{thm1}\eqref{thm1.2}.

\section{Action of projective functors on simple modules}\label{s4}

\subsection{Action of projective functors on simple modules}\label{s4.1}

Using Formula~\eqref{eq1}, the statement of Conjecture~\ref{conjecture2} has a more precise reformulation. 

{\bf Conjecture~\ref{conjecture2}'.}
{\em For $\mathfrak{g}=\mathfrak{sl}_n(\mathbb{C})$, let $x,y\in W$ be such that $x^{-1}\leq_L y$. 
Then $\theta_x L(y)$ is indecomposable.}

We note that the conjectured statement is not extendable outside type $A$. 
For example, $\theta_{st}L(ts)\cong \theta_{t}L(t)\oplus \theta_{t}L(tst)$ is decomposable in type $B_2$.

\subsection{$J$-comparable indices}\label{s4.2}

Our first observation is that Conjecture~\ref{conjecture2}' is true in the case 
$x\sim_J y$ by Lemma~\ref{uniquesummand}.

\subsection{Translation through a wall}\label{s4.3}

The following claim is fairly well-known to experts 
but we failed to find a proper reference.

\begin{proposition}\label{prop21}
Conjecture~\ref{conjecture2}' is true if $x$ is the longest element in some parabolic subgroup of $W$,
moreover, the corresponding statement is true for $\mathfrak{g}$ of any type. 
\end{proposition}

\begin{proof}
If  $x$ is the longest element in some parabolic subgroup, then $\theta_x$ is the translation 
through the intersection of walls which correspond to all simple reflections for 
this parabolic subgroup. For simplicity, we will simply say ``a wall'' instead of 
``intersection of all walls''.

As translation to a wall sends simple modules to simple modules 
or zero (because of $1$-dimensionality of the highest weight), 
by adjunction, translation from a wall (which is biadjoint to the translation to a wall) 
sends a simple module to a module with simple top, in 
particular, to an indecomposable module. Therefore translation through the wall, which is 
the composition of a translation to a wall and from a wall, sends a simple module to an
indecomposable module (or zero).
\end{proof}

\subsection{Projectives in $\mathcal{O}_0^{\mathfrak{p}}$}\label{s4.4}

Let $\mathfrak{p}$ be a parabolic subalgebra of $\mathfrak{g}$ containing  $\mathfrak{b}$. 
The module $P^{\mathfrak{p}}(e)$ has simple socle by \cite[Lemma~4.7]{MS0}.  
We denote this socle by $L(d_{\mathfrak{p}})$. If $\mathbf{R}$ is the right KL-cell in $W$
such that $\mathcal{O}_0^{\mathfrak{p}}=\mathcal{O}_0^{\hat{\mathbf{R}}}$, then 
$d_{\mathfrak{p}}$ is the Duflo involution in $\mathbf{R}$, see \cite[Corollary~3]{Ma1}.
The following claim is fairly well-known to experts 
but we failed to find a proper reference.

\begin{proposition}\label{prop22}
Conjecture~\ref{conjecture2}' is true if $y=d_{\mathfrak{p}}$ for some $\mathfrak{p}$,
moreover, the corresponding statement is true for $\mathfrak{g}$ of arbitrary type. 
\end{proposition}

\begin{proof}
Let $P$ be a multiplicity free projective generator of $\mathcal{O}_0^{\mathfrak{p}}$ 
and $Q$ be the maximal injective summand of $P$. Let $\mathtt{Q}$ be the opposite of the
endomorphism algebra of $Q$. By \cite[Theorem~10.1]{St2}, the functor 
$\mathrm{Hom}_{\mathcal{O}}(Q,{}_-)$ from $\mathcal{O}_0^{\mathfrak{p}}$
to $\mathtt{Q}$-mod is full and faithful on projective modules in 
$\mathcal{O}_0^{\mathfrak{p}}$.

Consider the trace $\mathrm{Tr}_Q(P)$ of $Q$ in $P$, that is the submodule of $P$ generated by all images
of $Q$ in $P$. Each endomorphism of $P$ restricts to $\mathrm{Tr}_Q(P)$, moreover, the
previous paragraph guarantees that this restriction map induces an isomorphism
\begin{displaymath}
\mathrm{End}_{\mathcal{O}}\big(P\big)\cong \mathrm{End}_{\mathcal{O}}\big(\mathrm{Tr}_Q(P)\big).
\end{displaymath}
Since projective functors are adjoint to projective functors and preserve projective-injective modules,
we have $\mathrm{Tr}_Q(\theta\, M)\cong \theta\,\mathrm{Tr}_Q(M)$ for any projective functor $\theta$
and any $M\in\mathcal{O}_0^{\mathfrak{p}}$. 

Put together, the above implies that the endomorphism algebra of the module
\begin{displaymath}
\theta_x L(d_{\mathfrak{p}}) =\theta_x \mathrm{Tr}_Q(P^{\mathfrak{p}}(e))
\cong \mathrm{Tr}_Q(\theta_x P^{\mathfrak{p}}(e))=
\mathrm{Tr}_Q(P^{\mathfrak{p}}(x))
\end{displaymath}
and the endomorphism algebra of  the module $P^{\mathfrak{p}}(x)$ are isomorphic. 
Therefore $\theta_x L(d_{\mathfrak{p}})$ is indecomposable as $P^{\mathfrak{p}}(x)$ is.
\end{proof}

\subsection{Tilting modules in $\mathcal{O}_0^{\mathfrak{p}}$}\label{s4.5}

Let $\mathfrak{p}$ be a parabolic subalgebra of $\mathfrak{g}$ containing  $\mathfrak{b}$,
$W_{\mathfrak{p}}$ the corresponding parabolic subgroup of $W$, $w_0^{\mathfrak{p}}$ the
longest element in $W_{\mathfrak{p}}$  and $w_{\mathfrak{p}}=w_0^{\mathfrak{p}}w_0$.
The following claim is fairly well-known to experts 
but we failed to find a proper reference.

\begin{proposition}\label{prop23}
Conjecture~\ref{conjecture2}' is true if $y=w_{\mathfrak{p}}$ for some $\mathfrak{p}$,
moreover, the corresponding  statement is true for $\mathfrak{g}$ of arbitrary type. 
\end{proposition}

\begin{proof}
Let $\mathfrak{q}$ be a parabolic subalgebra of $\mathfrak{g}$ containing  $\mathfrak{b}$
such that $w_0^{\mathfrak{q}}=w_0w_0^{\mathfrak{p}}w_0$.
Note that $w_0^{\mathfrak{p}}=w_0w_0^{\mathfrak{q}}$.

For a positive integer $i$, denote by $W_{\mathfrak{q}}^i$ the set of all elements in $W_{\mathfrak{q}}$
of length $i$ and write $m$ for the length of $w_0^{\mathfrak{q}}$.
By \cite[Section~4]{Lep}, the module $P^{\mathfrak{q}}(e)=\Delta^{\mathfrak{q}}(e)$ has a 
BGG type resolution of the following form:
\begin{equation}\label{eqnew123}
0\to\Delta(w_0^{\mathfrak{q}})\to
\bigoplus_{w\in W_{\mathfrak{q}}^{m-1}}\Delta(w)\to
\bigoplus_{w\in W_{\mathfrak{q}}^{m-2}}\Delta(w)\to
\dots\to \Delta(e)\to \Delta^{\mathfrak{q}}(e)\to 0.
\end{equation}

Consider the derived twisting functor $\mathcal{L}T_{w_0}$.
By \cite[Theorem~2.2]{AS}, the complex $\mathcal{L}T_{w_0}\,\Delta(w)$
is, in fact, a module, for all $w\in W$. For any reduced decomposition
$w_0=s_{1}s_2\cdots s_k$, we have
\begin{displaymath}
\mathcal{L}T_{w_0}\cong\mathcal{L}T_{s_1}\circ
\mathcal{L}T_{s_2}\circ\cdots\circ\mathcal{L}T_{s_k},
\end{displaymath}
see e.g. \cite[Remark~4.3(4)]{MS0}. In particular, we have
$\mathcal{L}T_{w_0}\cong \mathcal{L}T_{x}\circ\mathcal{L}T_{y}$, for any 
$x,y\in W$ such that $xy=w_0$ and $l(x)+l(y)=l(w_0)$. Note that we can write 
\begin{displaymath}
w_0=(w_0ww_0^{-1})(w_0w^{-1}) 
\end{displaymath}
and $l(w_0)=l(w_0ww_0^{-1})+l(w_0w^{-1})$.
Combining the above with \cite[Formula (2.3) and Theorem~2.3]{AS}, we see that
\begin{displaymath}
\begin{array}{rcl}
\mathcal{L}T_{w_0}\,\Delta(w)&\cong& 
\mathcal{L}T_{w_0ww_0^{-1}}\circ \mathcal{L}T_{w_0w^{-1}}\,\Delta(w)\\
&\cong& 
\mathcal{L}T_{w_0ww_0^{-1}}\,\Delta(w_0)\\
&\cong&\nabla(w_0w).\\
\end{array}
\end{displaymath}
Therefore, $\mathcal{L}T_{w_0}$ maps Verma modules to dual Verma modules.
Thus, applying $\mathcal{L}T_{w_0}$ to the resolution in \eqref{eqnew123}, 
produces a coresolution of $L(w_0^{\mathfrak{p}})$ by dual Verma modules
(see, for example, the proof of \cite[Proposition~4.4]{MS0} for details).

Applying $\theta_x$ to $P^{\mathfrak{q}}(e)$, produces (if non-zero) an indecomposable 
projective module in  $\mathcal{O}_0^{\mathfrak{q}}$. Since twisting commutes with projective functors,
see \cite[Theorem~3.2]{AS}, and, being a self-equivalence of the derived category of $\mathcal{O}_0$ 
by \cite[Corollary~4.2]{AS},  preserves indecomposability, the claim follows.
\end{proof}

The modules of the form $\theta_x\, L(w_{\mathfrak{p}})$ are exactly the indecomposable tilting modules
in the category $\mathcal{O}_0^{\mathfrak{p}}$. We refer the reader to \cite{MS0} and \cite{CM} for 
more details on the techniques used in the above proof and further results in this direction.

\subsection{General reduction to involutions}\label{s4.6}

\begin{proposition}\label{prop24}
Conjecture~\ref{conjecture2}' is true if and only if it is true for all $x$ and $y$
such that $y^2=e$.
\end{proposition}

\begin{proof}
Let $x,w\in W$ 
with $x^{-1}\leq_R w$ and let $y\in W$ be the Duflo involution in the left KL-cell of $w$. 
By \cite[Proposition~35]{MS}, the cell $2$-representations of $\cP$ corresponding to the 
right KL-cells $\mathcal{R}_w$ and $\mathcal{R}_y$ are equivalent and this equivalence swaps $L(w)$ and $L(y)$.
Thus, $\theta_x L(w)$ is indecomposable if and only if $\theta_x L(y)$ is. The claim follows.
\end{proof}

Another way to formulate Proposition~\ref{prop24} is to say that 
the property that Conjecture~\ref{conjecture2}' is true for all $x$
is an invariant of left KL-cells with respect to $y$.

\subsection{Connection to the double centralizer property}\label{s4.7}

Let $\mathbf{R}$ be a right KL-cell and $d\in \mathbf{R}$ the corresponding Duflo 
involution. Let $P$ be a multiplicity free projective generator  in 
$\mathcal{O}_0^{\hat{\mathbf{R}}}$  and $Q$ be the maximal injective summand of $P$.  
Let $\mathtt{Q}$ be the opposite of the endomorphism algebra of $Q$. 
The proof of Proposition~\ref{prop22} implies that, if
the functor $\mathrm{Hom}_{\mathcal{O}}(Q,{}_-)$ from $\mathcal{O}_0^{\hat{\mathbf{R}}}$
to $\mathtt{Q}$-mod is full and faithful on projective modules  in  $\mathcal{O}_0^{\hat{\mathbf{R}}}$
(this property is equivalent to a certain {\em double centralizer property}, see for example
\cite[Section~3.4]{Ma1}), then Conjecture~\ref{conjecture2}' 
is true  for $y=d$ and for any $x$.

By \cite[Theorem~11]{Ma1}, the condition that $\mathrm{Hom}_{\mathcal{O}}(Q,{}_-)$
is full and faithful on projective modules in  $\mathcal{O}_0^{\hat{\mathbf{R}}}$
is equivalent to the condition that Kostant's problem has the positive 
solution for $L(d)$ (see \cite{KM} for more details on Kostant's problem).
We refer the reader to \cite{KM,Ma3,Ka} for many examples of elements for which 
Kostant's problem has the positive solution.
However, as it is shown in  \cite{KM}, Kostant's problem can have the negative solution
for some $d$, even in type $A$ (the smallest example exists for the algebra $\mathfrak{sl}_4$).

\subsection{Further speculations}\label{s4.8}

Assume that we are in the situation as in the previous subsection, but such that 
there is no double centralizer property for our choice of $d$. Then the restriction map  
\begin{displaymath}
\mathrm{End}_{\mathcal{O}}\big(P\big)\to\mathrm{End}_{\mathcal{O}}\big(\mathrm{Tr}_Q(P)\big)
\end{displaymath}
is still injective but no longer surjective. By construction, the algebra 
$\mathrm{End}_{\mathcal{O}}\big(P\big)$ is positively graded, moreover,
$\mathrm{Tr}_Q(P)$ is a graded submodule of $P$. Therefore the algebra 
$\mathrm{End}_{\mathcal{O}}\big(\mathrm{Tr}_Q(P)\big)$ is a graded algebra.

To prove Conjecture~\ref{conjecture2}' , it would be sufficient to show that 
all homogeneous components of the graded quotient space 
\begin{displaymath}
\mathrm{End}_{\mathcal{O}}\big(\mathrm{Tr}_Q(P)\big)/ \mathrm{End}_{\mathcal{O}}\big(P\big)
\end{displaymath}
have strictly positive degrees. Indeed, in such a case these new components would only 
contribute to the Jacobson radical of $\mathrm{End}_{\mathcal{O}}\big(\mathrm{Tr}_Q(P)\big)$ 
and hence no essentially new idempotents can be created.

\noindent
Tobias Kildetoft, Department of Mathematics, Uppsala University,
Box 480,\\ SE-751~06, Uppsala, SWEDEN, {\tt tobias.kildetoft\symbol{64}math.uu.se}

\noindent
Volodymyr Mazorchuk, Department of Mathematics, Uppsala University,
Box 480, SE-751~06, Uppsala, SWEDEN, {\tt mazor\symbol{64}math.uu.se}

\end{document}